\documentclass[leqno,12pt]{article}
\usepackage{amsfonts}
\usepackage{amsmath, amssymb}
\usepackage{amsthm}

\theoremstyle{plain}
\newtheorem{theorem}{\indent\sc Theorem}[section]

\newtheorem{proposition}[theorem]{\indent\sc Proposition}

\theoremstyle{definition}

\newtheorem{remark}[theorem]{\indent\sc Remark}
\newtheorem{example}[theorem]{\indent\sc Example}

\newcommand\on{\operatorname}
\renewcommand\div{\on{div}}
\newcommand\Hess{\on{Hess}}
\newcommand\Ric{\on{Ric}}
\newcommand\scal{\on{scal}}

\newcommand\func{\operatorname}
\newcommand\grad{\func{grad}}

\begin{document}

\title{Remarks on Einstein solitons \\with certain types of potential vector field}
\author{Adara M. Blaga and Dan Radu La\c tcu}
\date{}
\maketitle

\centerline{\textit{\small{Dedicated to the memory of Professor Dr. Aurel Bejancu}}}

\begin{abstract}
We consider almost Einstein solitons $(V,\lambda)$ in a Riemannian manifold when $V$ is a gradient, a solenoidal or a concircular vector field. We explicitly express the function $\lambda$ by means of the gradient vector field $V$ and illustrate the result with suitable examples. Moreover, we deduce some geometric properties when the Ricci curvature tensor of the manifold satisfies certain symmetry conditions.
\end{abstract}

\markboth{{\small\it {\hspace{2cm} Remarks on Einstein solitons}}}{\small\it{Remarks on Einstein solitons
\hspace{2cm}}}


\bigskip

\section{Introduction}

In the 60', J. Eells and J. H. Sampson \cite{eesa} introduced the harmonic map flow which was the starting point of the theory of geometric flows.
Geometric flows deform in time a geometrical object (a metric or other map) by preserving certain properties of the initial one (such as the homotopy class).
For example, a Riemannian metric is being deformed by a geometric flow, such as the Ricci flow, in order to smooth out its singularities. Important in the study of the geometric flows are the solitons,  stationary solutions of the associated evolution equation.
In particular, Einstein solitons correspond to self-similar solutions of the nonlinear equation
\begin{equation}\label{g}
\frac{\partial }{\partial t}g(t)=-2\left(\Ric(g(t))-\frac{\scal}{2}g(t)\right),
\end{equation}
where $\Ric$ is the Ricci curvature tensor and $\scal$ is the scalar curvature of $g$. A special class of solutions for (\ref{g}) is constituted by gradient Einstein solitons, initially studied by G. Catino and L. Mazzieri \cite{cat}.

In this paper, we deduce some properties of almost Einstein solitons with gradient, solenoidal or torse-forming potential vector field. A central result is Theorem \ref{p1}, where for a gradient almost Einstein soliton, we explicitly compute the function $\lambda$ by means of $V$, without making all the laborious curvature computations.
We prove that if $V$ is a solenoidal vector field, then $|\Ric|^2\geq |\nabla V|^2$.
Notice that the solenoidal vector fields of gradient type are the gradients of harmonic functions, present in different branches of physical sciences.
We end our study by making some remarks on Ricci symmetric and Ricci semi-symmetric manifolds admitting almost Einstein solitons with concircular vector field. Properties of almost Riemann, Ricci and Yamabe solitons, as well as some of their generalizations, have been previously studied in \cite{bla, blaoz}.

\section{Gradient almost Einstein solitons}

On an $n$-dimensional smooth manifold $M$, a Riemannian metric $g$ and a non-vanishing vector field $V$ is said to define \textit{an Einstein soliton} \cite{ham} if there exists a real constant $\lambda$ such that
\begin{equation}\label{1}
\frac{1}{2}\pounds_{V}g+\func{Ric}=\left(\frac{\scal}{2}+\lambda\right) g,
\end{equation}
where $\pounds _{V}$ denotes the Lie derivative operator in the direction of the vector field $V$, $\Ric$ is the Ricci curvature tensor and $\scal$ is the scalar curvature of $g$.
If $\lambda$ is a smooth function on $M$, we say that $(V,\lambda)$ defines an \textit{almost Einstein soliton}. Moreover, if the vector field $V$ is of gradient type, we say that $(V,\lambda)$ defines a \textit{gradient almost Einstein soliton}.

Taking the trace of the soliton equation (\ref{1}), we obtain
\begin{equation}\label{9}
\scal=\frac{2}{n-2}[\div(V)-n\lambda],
\end{equation}
provided $n\geq 3$.
If $n=2$, then $\lambda=\frac{\div(V)}{2}$.

For $n\geq 3$, (\ref{1}) becomes
\begin{equation}\label{k}
\frac{1}{2}\pounds _{V}g+\func{Ric}=\frac{1}{n-2}[\div(V)-2\lambda]g.
\end{equation}

If $V$ is a solenoidal vector field, i.e. $\div(V)=0$, then
$$\lambda=-\frac{n-2}{2n}\scal.$$

From the Bochner formula, for any vector field $V$ on $M$ of gradient type, we have \cite{blag}:
\begin{equation}\label{15}
\Ric(V,V)=\frac{1}{2}\Delta(|V|^2)-|\nabla V|^2-V(\div(V)).
\end{equation}

If $n\geq 3$, replacing from (\ref{9}) the scalar curvature in (\ref{1}) and computing the relation in $(V,V)$, we have:
\begin{equation}\label{43}
\Ric(V,V)=-\frac{1}{2}V(|V|^2)+\frac{1}{n-2}|V|^2\cdot \div(V)-\frac{2}{n-2}|V|^2\cdot \lambda.
\end{equation}

In view of (\ref{15}) and (\ref{43}), we can state:

\begin{theorem}\label{p1}
If $(V,\lambda)$ defines a gradient almost Einstein soliton on the $n$\text{-}di\-men\-sio\-nal Riemannian manifold $(M,g)$, $n\geq 3$, then
the function $\lambda$ can be explicitly expressed in terms of $V$:
$$\lambda=-\frac{n-2}{4|V|^2}[\Delta(|V|^2)-2|\nabla V|^2+V(|V|^2)-2V(\div(V))]+\frac{1}{2}\div(V).$$
\end{theorem}

Now we verify this result with two examples.

\begin{example}
Let $M=\{(x,y,z)\in\mathbb{R}^3| z>0\}$, where $(x,y,z)$ are the standard coordinates in $\mathbb{R}^3$, and consider the Riemannian metric:
$$g:=\frac{1}{z^2}(dx\otimes dx+dy\otimes dy+dz\otimes dz).$$

Then an orthonormal frame field on $TM$ is given by:
$$E_1=z\frac{\partial}{\partial x}, \ \ E_2=z\frac{\partial}{\partial y}, \ \ E_3=z\frac{\partial}{\partial z}.$$

By Koszul's formula, we determine the components Levi-Civita connection:
$$\nabla_{E_1}E_1=E_3, \ \ \nabla_{E_1}E_2=0, \ \ \nabla_{E_1}E_3=-E_1,$$
$$\nabla_{E_2}E_1=0, \ \ \nabla_{E_2}E_2=E_3, \ \ \nabla_{E_2}E_3=-E_2,$$
$$\nabla_{E_3}E_1=0, \ \ \nabla_{E_3}E_2=0, \ \ \nabla_{E_3}E_3=0.$$

Consequently, the components of the Riemann and Ricci curvatures are:
$$R(E_1,E_2)E_2=-E_1=R(E_1,E_3)E_3, \ \ R(E_2,E_1)E_1=-E_2=R(E_2,E_3)E_3,$$
$$R(E_3,E_1)E_1=-E_3=R(E_3,E_2)E_2, \ \ R(E_i,E_j)E_k=0, \ \ \text{for} \ \ i\neq j, j\neq k, k\neq i,$$
$$\Ric(E_i,E_i)=-2, \ \ \Ric(E_i,E_j)=0,  \ \ \text{for} \ \ i\neq j$$
and the scalar curvature of $M$ equals to $-6$.

\pagebreak

For $V=\frac{1}{z}E_3$, the Lie derivative of $g$ in the direction of $V$ is given by:
$$(\pounds _{V}g)(E_1,E_1)=(\pounds _{V}g)(E_2,E_2)=(\pounds _{V}g)(E_3,E_3)=-\frac{2}{z}$$
and from the Einstein soliton's equation we obtain that
the pair $(V=\frac{\partial}{\partial z}, \lambda=-\frac{1}{z}+1)$ defines a gradient almost Einstein soliton, with $V=\grad(f)$, for $f(x,y,z)=-\frac{1}{z}$.

On the other hand, one can check that $|V|^2=\frac{1}{z^2}$, $V(|V|^2)=-\frac{2}{z^3}$, $\Delta(|V|^2)=\frac{8}{z^2}$, $|\nabla V|^2=\frac{3}{z^2}$, $\div(V)=-\frac{3}{z}$, $V(\div(V))=\frac{3}{z^2}$ therefore, $\lambda=-\frac{1}{z}+1$ is obtained from Theorem \ref{p1}, without making all the above laborious curvature computations.
\end{example}

\begin{example}
Let $M=\{(x,y,z)\in\mathbb{R}^3| z>0\}$, where $(x,y,z)$ are the standard coordinates in $\mathbb{R}^3$, and consider the Riemannian metric:
$$g:=e^{2z}(dx\otimes dx+dy\otimes dy)+dz\otimes dz.$$

Then an orthonormal frame field on $TM$ is given by:
$$E_1=e^{-z}\frac{\partial}{\partial x}, \ \ E_2=e^{-z}\frac{\partial}{\partial y}, \ \ E_3=\frac{\partial}{\partial z}.$$

By Koszul's formula, we determine the components Levi-Civita connection:
$$\nabla_{E_1}E_1=-E_3, \ \ \nabla_{E_1}E_2=0, \ \ \nabla_{E_1}E_3=E_1,$$
$$\nabla_{E_2}E_1=0, \ \ \nabla_{E_2}E_2=-E_3, \ \ \nabla_{E_2}E_3=E_2,$$
$$\nabla_{E_3}E_1=0, \ \ \nabla_{E_3}E_2=0, \ \ \nabla_{E_3}E_3=0.$$

Consequently, the components of the Riemann and Ricci curvatures are:
$$R(E_1,E_2)E_2=-E_1=R(E_1,E_3)E_3, \ \ R(E_2,E_1)E_1=-E_2=R(E_2,E_3)E_3,$$
$$R(E_3,E_1)E_1=-E_3=R(E_3,E_2)E_2, \ \ R(E_i,E_j)E_k=0, \ \ \text{for} \ \ i\neq j, j\neq k, k\neq i,$$
$$\Ric(E_i,E_i)=-2, \ \ \Ric(E_i,E_j)=0,  \ \ \text{for} \ \ i\neq j$$
and the scalar curvature of $M$ equals to $-6$.

For $V=e^zE_3$, the Lie derivative of $g$ in the direction of $V$ is given by:
$$(\pounds _{V}g)(E_1,E_1)=(\pounds _{V}g)(E_2,E_2)=(\pounds _{V}g)(E_3,E_3)=2e^z$$
and from the Einstein soliton's equation we obtain that
the pair $(V=e^z\frac{\partial}{\partial z}, \lambda=e^z+1)$ defines a gradient almost Einstein soliton, with $V=\grad(f)$, for $f(x,y,z)=e^z$.

On the other hand, one can check that $|V|^2=e^{2z}$, $V(|V|^2)=2e^{3z}$, $\Delta(|V|^2)=8e^{2z}$, $|\nabla V|^2=3e^{2z}$, $\div(V)=3e^z$, $V(\div(V))=3e^{2z}$ therefore, $\lambda=e^z+1$ is immediately obtained from Theorem \ref{p1}.
\end{example}

If $V=\grad{(f)}$, with $f$ a smooth function on $M$, equation (\ref{1}) becomes
\begin{equation}\label{e1}
\Hess(f)+\Ric=\left(\frac{\scal}{2}+\lambda\right)g
\end{equation}
and (\ref{9}) becomes
\begin{equation}\label{e3}
\scal=\frac{2}{n-2}[\Delta(f)-n\lambda]
\end{equation}
which by differentiating gives
\begin{equation}\label{e4}
d(\Delta(f))=\frac{n-2}{2}  d(\scal)+nd\lambda
\end{equation}
and by taking the gradient of (\ref{e3})
\begin{equation}\label{e6}
\grad(\Delta(f))=\frac{n-2}{2}  \grad(\scal)+n\grad(\lambda).
\end{equation}

Applying the divergence operator to (\ref{e1}) and using Schur's lemma, we obtain
\begin{equation}\label{he}
\div(\Hess(f))=d\lambda.
\end{equation}

Using the relation proved in \cite{blag}, namely
\begin{equation}\label{hh}
\div(\Hess(f))=d(\Delta(f))+i_{Q(\grad(f))}g,
\end{equation}
where $i$ denotes the interior product and $Q$ is the Ricci operator defined by $g(QX,Y):=\Ric(X,Y)$, from (\ref{he}) and (\ref{hh}) we get
\begin{equation}\label{e10}
d(\Delta(f))=d\lambda-i_{Q(\grad(f))}g.
\end{equation}

Comparing (\ref{e4}) and (\ref{e10}) and using (\ref{e6}), we obtain:
\begin{proposition}
If $(V=\grad(f),\lambda)$ defines a gradient almost Einstein soliton on the $n$\text{-}di\-men\-sio\-nal Riemannian manifold $(M,g)$, $n\geq 3$, then
\begin{equation}\label{ll}
Q(\grad(f))=-(n-1)\grad(\lambda)-\frac{n-2}{2}\grad(\scal).
\end{equation}
Moreover, if $\grad(f)\in \ker (Q)$, then
$$\grad(\lambda)=\grad(\Delta(f))=-\frac{n-2}{2(n-1)}\grad(\scal).$$
\end{proposition}

In equation (\ref{k}), by taking the scalar product with $\Ric$, we get
$$\langle\pounds _{V}g,\Ric\rangle=-2|\Ric|^2+\frac{4}{(n-2)^2}[(\div(V))^2-(n+2)\div(V)\cdot\lambda+2n\lambda^2]$$
and by taking the scalar product with $\pounds _{V}g$, we get
$$\langle\pounds _{V}g,\Ric\rangle=-\frac{1}{2}|\pounds _{V}g|^2+\frac{2}{n-2}[(\div(V))^2-2\div(V)\cdot\lambda].$$

Since $|\pounds _{V}g|^2=4|\nabla V|^2$, if $V$ is a solenoidal vector field, then
$$|\nabla V|^2=|\Ric|^2-\frac{4n}{(n-2)^2}\lambda^2$$
which implies

\begin{proposition}
If $(V,\lambda)$ defines an almost Einstein soliton on the $n$\text{-}di\-men\-sio\-nal Riemannian manifold $(M,g)$, $n\geq 3$, with solenoidal potential vector field $V$, then
$$|\Ric|^2\geq |\nabla V|^2.$$
\end{proposition}

\section{Almost Einstein solitons with concircular potential vector field}

Assume next that $V$ is a torse-forming vector field, i.e. $\nabla V=aI+\psi\otimes V$, where $a$ is a smooth function, $\psi$ is a $1$-form, $\nabla$ is the Levi-Civita connection of $g$ and $I$ is the identity endomorphism on the space of vector fields.
Then \cite{bla}
$$\div(V)=na+\psi(V),$$
$$
\pounds _{V}g=2ag+\psi\otimes \theta+\theta\otimes \psi,
$$
where $\theta:=i_{V}g$ and from (\ref{k}), we get
\begin{equation}\label{7}
\Ric=\frac{1}{n-2}[\psi(V)+2(a-\lambda)]g-\frac{1}{2}(\psi\otimes \theta+\theta\otimes \psi),
\end{equation}
$$Q=\frac{1}{n-2}[\psi(V)+2(a-\lambda)]I-\frac{1}{2}(\psi\otimes V+\theta\otimes \zeta),$$
hence
$$\scal=\frac{2}{n-2}[\psi(V)+n(a-\lambda)],$$
where $\psi:=i_{\zeta}g$.

\pagebreak

Computing the Riemann curvature, for $\nabla V=aI+\psi\otimes V$, we get
$$R(X,Y)V=(da-a\psi)(X)Y-(da-a\psi)(Y)X+[(\nabla_X\psi)Y-(\nabla_Y\psi)X]V,$$
for any $X,Y\in \mathfrak{X}(M)$.

Moreover, if $\psi$ is a Codazzi tensor field, i.e. $(\nabla_X\psi)Y=(\nabla_Y\psi)X$, for any $X,Y\in\mathfrak{X}(M)$, then
\begin{equation}\label{5}
\func{Ric}(V,V)=(1-n)[V(a)-a\psi(V)].
\end{equation}

But if $(V,\lambda)$ defines an almost Einstein soliton with torse-forming potential vector
field $V$, from (\ref{7}) we have
\begin{equation}\label{6}
\func{Ric}(V,V)=\frac{|V|^2}{n-2}[2(a-\lambda)-(n-3)\psi(V)],
\end{equation}%
and we can state:

\begin{proposition}
\label{p} If $(V,\lambda)$ defines an almost Einstein
soliton on the $n$-di\-men\-sional Riemannian manifold $(M,g)$, $n\geq 3$, such that the potential vector field $V$ is torse-forming and $\psi$ is a Codazzi tensor field, then
$$\lambda=a+\frac{(n-1)(n-2)}{2|V|^2}V(a)-\frac{1}{2|V|^2}[(n-1)(n-2)a+(n-3)|V|^2]\psi(V).$$

Moreover, if $V$ is a concircular vector field, $\nabla V=aI$ and $a$ is a non zero constant, then

i) $\lambda=a$;

ii) $M$ is a Ricci-flat manifold (i.e. $\Ric=0$).
\end{proposition}

\begin{proof}
From (\ref{5}) and (\ref{6}) we get the expression of $\lambda$.

If $V$ is a concircular vector field and $a$ is a non zero constant, then
$\lambda=a$ and from (\ref{7}) we get $\Ric=0$.
\end{proof}

\begin{remark}
Let $(V,\lambda)$ define an almost Einstein soliton on $(M,g)$, $n\geq 3$, with $V$ torse-forming and $g$-orthogonal to $\zeta$, and $\psi$ a Codazzi tensor field.

i) Then $M$ is an almost quasi-Einstein manifold with associated functions $$\left(-(n-1)\frac{V(a)}{|V|^2},-\frac{1}{2}\right).$$

ii) Moreover, if $V$ is a concircular vector field, then $M$ is an almost Einstein manifold with associated function $$-(n-1)\frac{V(a)}{|V|^2}.$$
\end{remark}

Now differentiating covariantly (\ref{7}), we get
$$
(\nabla_X\Ric)(Y,Z)=\frac{1}{n-2}[X(\psi(V))+2X(a-\lambda)]g(Y,Z)
$$
\begin{equation}\label{44}
-\frac{1}{2}[\psi(Y)(\nabla_X\theta)Z+\psi(Z)(\nabla_X\theta)Y+\theta(Y)(\nabla_X\psi)Z+\theta(Z)(\nabla_X\psi)Y],
\end{equation}
for any $X,Y,Z\in \mathfrak{X}(M)$ and we can state:
\begin{proposition}
If $(V,\lambda)$ defines an almost Einstein soliton on $(M,g)$ with concircular potential vector field $V$, $\nabla V=aI$ and $a$ is a smooth function, then

i) $\nabla \Ric=0$ if and only if $d\lambda=da$;

ii) $\nabla \Ric =\theta\otimes \Ric$ if and only if $$\grad(\lambda-a)=(\lambda-a)V;$$

iii) $(\nabla_X\Ric)(Y,Z)=(\nabla_Y\Ric)(X,Z)$, for any $X,Y,Z\in \mathfrak{X}(M)$, if and only if $$d(\lambda-a)\otimes I=I\otimes d(\lambda-a);$$

iv) $(\nabla_X\Ric)(Y,Z)+(\nabla_Y\Ric)(Z,X)+(\nabla_Z\Ric)(X,Y)=0$, for any $X,Y,Z\in \mathfrak{X}(M)$, implies
$$\grad(\lambda-a)=-2V(\lambda-a)\frac{V}{|V|^2}.$$
\end{proposition}
\begin{proof}
If $V$ is a concircular vector field, from (\ref{44}) we get
$$(\nabla_X\Ric)(Y,Z)=\frac{2}{n-2}X(a-\lambda)g(Y,Z),$$
for any $X,Y,Z\in \mathfrak{X}(M)$ and we easily get the conclusions.
\end{proof}

\bigskip

\textit{Adara M. Blaga}


\textit{West University of Timi\c{s}oara}

\textit{Bld. V. P\^{a}rvan nr. 4, 300223, Timi\c{s}oara, Rom\^{a}nia}

\textit{adarablaga@yahoo.com}

\bigskip

\textit{Dan Radu La\c tcu}


\textit{West University of Timi\c{s}oara}

\textit{Bld. V. P\^{a}rvan nr. 4, 300223, Timi\c{s}oara, Rom\^{a}nia}

\textit{latcu07@yahoo.com}

\end{document}